\documentclass[10pt]{amsart}
\usepackage[utf8]{inputenc}
\usepackage{amsmath,amssymb, amsthm}
\usepackage{mathtools}
\usepackage[margin=1in]{geometry}
\usepackage{multicol}
\usepackage{fancyhdr}
\usepackage{pgf,tikz}
\usetikzlibrary{arrows}
\definecolor{uququq}{rgb}{0.25,0.25,0.25}
\definecolor{qqqqff}{rgb}{0,0,1}
\usepackage{array}
\usepackage{enumitem}
\usepackage{newclude}
\usepackage{dirtytalk}
\usepackage{asymptote}
\usepackage{tabto}
\usepackage{graphicx}
\usepackage[toc,page]{appendix}
\usepackage{indentfirst}
\graphicspath{{Logos/}}
\newtheorem{theorem}{Theorem}
\newtheorem{lemma}{Lemma}
\newtheorem{remark}{Remark}
\theoremstyle{definition}
\newtheorem{definition}{Definition}
\usepackage{hyperref}

\begin{document}
\begin{center}
    \Large
    \textbf{On the Distance Spectra of Extended Double Stars}
        
    \vspace{0.4cm}
    \large
    Anuj Sakarda, Jerry Tan, Armaan Tipirneni
    \vspace{0.9cm}
    
    \textbf{Abstract}
\end{center}
The distance matrix of a connected graph is defined as the matrix in which the entries are the pairwise distances between vertices. The distance spectrum of a graph is the set of eigenvalues of its distance matrix. A graph is said to be determined by its distance spectrum if there does not exist a non-isomorphic graph with the same spectrum. The question of which graphs are determined by their spectrum has been raised in the past, but it remains largely unresolved. In this paper, we prove that extended double stars are determined by their distance spectra.

\tableofcontents
\begin{center}
    \large
    \textbf{Keywords}
\end{center}
Distance spectra, distance matrix, spectral graph theory, cospectral graph, determined by spectrum, graph theory, interlacing.

\section{Introduction}
\subsection{Background}
Graphs in this paper are assumed to be undirected, simple and connected. For a graph $G$, we denote its vertex set $V(G)$. The distance between vertices $u$ and $v$ of a graph $G$ is denoted by $d_{uv}$. The diameter of $G$, denoted by $d(G)$, is the maximum distance between any pair of vertices of $G$. For $S \subseteq V(G)$, we denote by $G[S]$ the subgraph of $G$ induced by the vertices in $S$, which is the subgraph of $G$ whose vertex set is $S$ and whose edge set consists of all edges of $G$ that have both ends in $S.$

The distance matrix of a graph is defined as follows:

\begin{definition}[Distance Matrix]
The \emph{distance matrix} $D(G)$ of a connected graph $G$ with vertices $v_1, v_2, \dots, v_n$ is the $n \times n$ matrix with $ij-$entry $d_{ij}$, where $d_{ij}$ is the distance between $v_i$ and $v_j$.
\end{definition}

We refer to the \textit{distance spectrum} of a graph as the set of eigenvalues of its distance matrix. Further, we index the eigenvalues of $D(G)$ as $\lambda_1(D)\geq \lambda_2(D)\geq \dots \geq \lambda_n(D).$ We define the \textit{distance characteristic polynomial} of $G$ as $\det(D(G)-\lambda I),$ where $I$ is the identity matrix.

Two graphs are \textit{cospectral} if they have the same distance spectrum, and We refer to a graph as determined by its distance spectrum if there exists no non-isomorphic graph cospectral to it. 

\begin{definition}[Extended Double Stars]
Let $a,b$ be non-negative integers. As defined in \cite{xue}, a double star, denoted by $S(a,b)$, is a graph formed by joining the centers of the complete bipartite graphs $K_{1,a}$ and $K_{1,b}$ via an edge. An \emph{extended double star}, denoted by $T(a,b),$ is the graph formed by joining the centers of $K_{1,a}$ and $K_{1,b}$ to a common vertex.
\end{definition}

Note that $T(a,b)$ has $a+b+3$ vertices. This paper explores whether the graphs $T(a,b)$ are determined by their distance spectra. 

\subsection{Previous Work and Our Results}

The question of which graphs are determined by their distance spectra has been raised in \cite{vandam}, though this question is largely open.

Previous work in \cite{xue} has proven that $S(a,b)$ is determined by its distance spectrum which provided the inspiration for this paper. Other work in \cite{lin} includes proving that $K_n$ and the complete split graph are determined by their distance spectra. In addition, in \cite{jin} and \cite{lu}, complete $k$-partite graphs, graphs with exactly two distance eigenvalues not equal to $-1$ or $-3$, and the friendship graph have been proven to be determined by their distance spectra. Very recently, a theorem was proven in \cite{zhang} that determines which graphs with at most 3 distance eigenvalues not equal to $-1$ or $-2$ are determined by their distance spectra. The distance spectra of the cycle graphs $C_n$ have been determined in \cite{atik} and \cite{fowler}. The majority of the results surrounding the spectra of graphs, however, have focused on non-distance matrices of a graph such as the adjacency matrix, the Laplacian matrix, the signless Laplacian matrix, and generalized adjacency matrices, which take the form $xI+yJ+zA$ for $x,y,z\in\mathbb{R},z\neq0,$ where $J$ is the matrix of all $1$s, $I$ is the identity matrix, and $A$ is the adjacency matrix. Graphs of diameter $2$ follow the property that $D = 2J - 2I - A$, so the distance matrix is a generalized adjacency matrix, which is studied in \cite{koolen}. 

This paper investigates a family of graphs of diameter $4$, $T(a,b),$ and proves it is determined by its distance spectrum. Our proof is split into cases based on diameter and show that no non-isomorphic graphs have the same distance spectrum as $T(a,b)$. 

\section{Preliminaries}

We introduce the following theorem derived from the Cauchy Interlace Theorem, which is also presented in \cite{vandam}. This theorem functions as an essential foundation of the proof presented in this paper, as it is used to prove key tools later.

\begin{theorem}
\label{interlacing}
Let G be a graph with $n$ vertices and distance matrix $D(G)$. Denote the $n$ eigenvalues as $\lambda_1(D(G)), \lambda_2(D(G)), \dots , \lambda_n(D(G))$ where $\lambda_1(D(G))$ $\geq \lambda_2(D(G)) \geq \dots \geq \lambda_n(D(G))$. Let $H$ be an induced subgraph of $G$ with $m$ vertices with distance spectrum $\mu_1(D(H))$ $\geq \mu_2(D(H)) \dots \geq \mu_m(D(H))$. If $D(H)$ is a principal submatrix of $D(G)$, then $\lambda_{n-m+i}(D(G)) \leq \mu_i(D(H)) \leq \lambda_i(D(G))$ for $i = 1, 2, \dots, m$.
\end{theorem}

\begin{lemma}
\label{double star polynomial}
Let $G=T(a,b)$. Then $G$ has distance characteristic polynomial 
$$(-\lambda-2)^{a+b-2}p_{a,b}(\lambda),$$

\noindent where $p_{a,b}(\lambda)=16 + 8a + 8b + 40\lambda + 36a\lambda + 36b\lambda + 8ab\lambda + 28\lambda^2 + 44 a\lambda^2 + 44b\lambda^2 + 24 ab\lambda^2 + 2 \lambda^3 + 18 a\lambda^3 + 18 b\lambda^3 +   12 ab\lambda^3 - 4 \lambda^4 + 2 a\lambda^4 + 2 b\lambda^4 - \lambda^5.$ When $a$ and $b$ are clear from context we simply use $p(\lambda)=p_{a,b}(\lambda).$
\end{lemma}

\begin{proof}
Let $J$ denote the all-ones matrix.

Note that $T(a,b)$ has distance matrix $$D(G)=\begin{pmatrix} 0 & 1 & 2 & J_{1\times a} & 3J_{1\times b}  \\ 1 & 0 & 1 & 2J_{1\times a} & 2J_{1\times b} \\ 2 & 1 & 0 & 3J_{1\times a} & J_{1\times b} \\ J_{a\times 1} & 2J_{a\times 1} & 3J_{a\times 1} & 2J_{a\times a}-2I & 4J_{a\times b} \\ 3J_{b\times 1} & 2J_{b\times 1} & J_{b\times 1} & 4J_{b\times a} & 2J_{b\times b} - 2I \end{pmatrix}.$$

Then the characteristic polynomial is \begin{align*}\det(D(G)-\lambda I)&=\begin{vmatrix} -\lambda & 1 & 2 & J_{1\times a} & 3J_{1\times b}  \\ 1 & -\lambda & 1 & 2J_{1\times a} & 2J_{1\times b} \\ 2 & 1 & -\lambda & 3J_{1\times a} & J_{1\times b} \\ J_{a\times 1} & 2J_{a\times 1} & 3J_{a\times 1} & 2J_{a\times a}-(2+\lambda)I & 4J_{a\times b} \\ 3J_{b\times 1} & 2J_{b\times 1} & J_{b\times 1} & 4J_{b\times a} & 2J_{b\times b} - (2+\lambda)I \end{vmatrix}\\&=(-\lambda-2)^{a+b-2}\begin{vmatrix} -\lambda & 1 & 2 & a & 3b  \\ 1 & -\lambda & 1 & 2a & 2b \\ 2 & 1 & -\lambda & 3a & b \\ 1 & 2 & 3 & 2a-2-\lambda & 4b \\ 3 & 2 & 1 & 4a & 2b - 2 - \lambda \end{vmatrix} \\ &= (-\lambda-2)^{a+b-2} (16 + 8a + 8b + 40\lambda + 36a\lambda + 36b\lambda + 8ab\lambda + 28\lambda^2 + 44 a\lambda^2 + 44b\lambda^2 \\ &\quad+ 24 ab\lambda^2 + 2 \lambda^3 + 18 a\lambda^3 + 18 b\lambda^3 +   12 ab\lambda^3 - 4 \lambda^4 + 2 a\lambda^4 + 2 b\lambda^4 - \lambda^5).\end{align*}
\end{proof}

In particular, we find that $T(1,1)$ has distance spectrum $$\{8.2882, -0.5578, -0.7639, -1.7304, -5.2361\}.$$

If $a=b=0,$ then $T(a,b)=P_3$ is determined by its spectrum by \cite{xue}. Next, if $a>b=0$ then $T(a,b)=S(a,1),$ which is determined by its spectrum by \cite{xue}. A similar argument holds if $b>a=0.$

Then for $a,b \geq 1$, $D(T(1,1))$ is a principal submatrix of $D(T(a,b))$. Then by Theorem~\ref{interlacing},
\begin{align*}
    &\lambda_1(D(T(a,b)))\geq\lambda_1(D(T(1,1))=8.2882,\\
    &\lambda_2(D(T(a,b)))\geq\lambda_2(D(T(1,1))=-0.5578,\\
    &\lambda_3(D(T(a,b)))\geq\lambda_3(D(T(1,1))=-0.7639,\\
    &\lambda_4(D(T(a,b)))\geq\lambda_4(D(T(1,1))=-1.7304,\\
    &\lambda_n(D(T(a,b)))\leq\lambda_5(D(T(1,1))=-5.2361.
\end{align*}

Let $c=\max(a,b),$ so $c\geq1.$ Then $T(c,c)$ has polynomial 
\begin{align*}&(-\lambda-2)^{2c-2}(\lambda^5+(4c-4)\lambda^4+(12c^2+36c+2)\lambda^3+(24c^2+88c+28)\lambda^2+(8c^2+72c+40)\lambda+16c+16)\\&=(-\lambda-2)^{2c-2}(\lambda^2+(2c+4)\lambda+4)(-\lambda^3+6c\lambda^2+(12c+6)\lambda+4c+4).
\end{align*}

We denote $f(\lambda)=-\lambda^3+6c\lambda^2+(12c+6)\lambda+4c+4$ and $g(\lambda)=\lambda^2+(2c+4)\lambda+4$ for convenience.

For $c=1$, $\lambda=-1.7304, -0.5578, 8.2882$ are the zeros of $f(\lambda)$. Note that $f(-1.7304)=6(c-1)(-1.7304)^2+12(c-1)(-1.7304)+4(c-1)=1.20(c-1)>0$ for $c>1.$ 

Let $f'(\lambda)=6\lambda^2+12\lambda+4.$ Note $f'(\lambda)$ has zeros $\lambda=-1.5774, -0.4226.$ Since $f(\lambda)=cf'(\lambda)-\lambda^3+6\lambda+4,$ we have $f(-1.5774)=-(-1.5774)^3+6(-1.5774)+4<0.$ Thus, by the Intermediate Value Theorem, $f(\lambda)$ has a root in $[-1.7304, -1.5774)$ for $c \geq 1.$

By similar arguments, $f(\lambda)$ has zeros in the following intervals for $c\geq1$ : $$[-1.7304, -1.5774); [-0.5578, -0.4226); [8.2882, \infty).$$ 

Next, $g(\lambda)$ has zeros $-(c+2)\pm\sqrt{c^2+4c}$ which lie in $(-\infty, -5.2361]$ and $[-0.7639, 0)$ for $c\geq1.$

Clearly, $D(T(a,b))$ is a principal submatrix of $D(T(c,c))$, so by Theorem~\ref{interlacing},

\begin{align*}
    &\lambda_2(D(T(a,b)))\leq\lambda_2(D(T(c,c))<0,\\
    &\lambda_3(D(T(a,b)))\leq\lambda_3(D(T(c,c))<-0.4226,\\
    &\lambda_4(D(T(a,b)))\leq\lambda_4(D(T(c,c))<-1.5774.\\
\end{align*}

We can see that $-2,$ which must have multiplicity $a+b-2 = n-5,$ does not lie in any of the above intervals, so $\lambda_i=-2$ for $5\leq i \leq n-1.$ The complete spectrum for $T(a,b), a,b\geq 1$ is shown below. 

\begin{center}
\setlength\tabcolsep{3pt}
\begin{tabular}{ |c|c|c|c|c|c|c|c|c|c|c| } 
 \hline
  $\lambda_1$ & $\lambda_2$ & $\lambda_3$ & $\lambda_4$ & $\lambda_5$ & $\cdots$ & $\lambda_{n-1}$ & $\lambda_n$\\ 
 \hline
 $[8.2882, \infty)$ & $[-0.5578,0)$ & $[-0.7639, -0.4226)$ & $[-1.7304,-1.5774)$ & -2 & $\cdots$ & -2 & $(-\infty,-5.2361]$\\ 
 \hline
\end{tabular}.
\end{center}

\textbf{Note:} From this point on we denote $G$ as some graph cospectral to $T(a,b)$, assuming such a graph exists.

\begin{lemma}
\label{submatrix}
Let $D'_{m\times m}$ be a principal submatrix of $D(T(a,b))$. Then
\begin{itemize}
    \item $\lambda_2(D') < 0,$ $\lambda_3(D') < -0.4226,$ and $\lambda_4(D') < -1.5774;$ 
    \item $\lambda_i(D') = -2$ for $5\leq i \leq m-1.$
\end{itemize}
\end{lemma}

We have the following result from \cite{atik} and \cite{fowler}.
\begin{lemma}
\label{cycle spectrum}
The cycle graph $C_{n}$ has distance eigenvalues as follows.

For $n=2p+1$ where $p\in\mathbb{Z},$ we have $\frac{n^{2}-1}{4},\left(-\frac{1}{4} \sec ^{2}\left(\frac{\pi j}{n}\right)\right)^{(2)}, j=1, \dots, p.$

For $n=2p$ where $p\in\mathbb{Z}$, we have $\frac{n^{2}}{4}, 0^{(p-1)},\left(-\csc ^{2}\left(\frac{\pi(2 j-1)}{n}\right)\right)^{(2)}, j=1, \ldots,\left\lfloor\frac{p}{2}\right\rfloor$ and $-1$ if $p$ is
odd. 
\end{lemma}

\noindent\textbf{Notation:} We denote by $D_G[S]$ the principal submatrix of $D(G)$ induced by $S.$

\begin{definition}[Forbidden Subgraph]
A graph $H$ is a \emph{forbidden subgraph} of a graph $G$ if the set of induced subgraphs of $G$ does not include a graph isomorphic to $H$.
\end{definition}

\begin{lemma}
\label{simple cycles}
For all $n \geq 4,$ if $D_G[V(C_n)]=D(C_n),$ then $C_n$ is a forbidden subgraph of $G$.
\end{lemma}

\begin{proof}
We proceed with contradiction. First consider $n=4,5,6,7.$

Suppose $C_4$ is a subgraph of $G$ and $D_G[V(C_4)] = D(C_4).$ Then $\lambda_2(D(C_4)) = 0,$ contradicting Lemma~\ref{submatrix}. So $C_4$ is forbidden. 

Suppose $C_5$ is a subgraph of $G$ and $D_G[V(C_5)] = D(C_5).$ Then $\lambda_3(D(C_5)) = -0.3820$ contradicting Lemma~\ref{submatrix}. So $C_5$ is forbidden. 

Suppose $C_6$ is a subgraph of $G$ and $D_G[V(C_6)] = D(C_6).$ Then $\lambda_5(D(C_6)) \neq -2$ contradicting Lemma~\ref{submatrix}. So $C_6$ is forbidden. 

Suppose $C_7$ is a subgraph of $G$ and $D_G[V(C_7)] = D(C_7).$ Then $\lambda_5(D(C_7)) \neq -2$ contradicting Lemma~\ref{submatrix}. So $C_7$ is forbidden. 

Suppose $C_n$ is a subgraph of $G$ and $D_G[V(C_n)] = D(C_n)$ where $n\geq 8$. 

Using Lemma~\ref{cycle spectrum}, we count the multiplicity of $-2$ as an eigenvalue of $C_n$. We check that $\frac{n^2-1}{4}, \frac{n^2}{4}, 0, -1 \neq -2.$ Next, $\left(-\frac{1}{4} \sec ^{2}\left(\frac{\pi j}{n}\right)\right) = -2$ holds for at most one value of $j,$ and the same is true for $\left(-\csc ^{2}\left(\frac{\pi(2 j-1)}{n}\right)\right)$. Thus the multiplicity of $-2$ as an eigenvalue of $C_n$ is at most $2.$ On the other hand, by Lemma~\ref{submatrix}, $C_n$ must have $\lambda_i = -2$ for $5\leq i\leq n-1$, so $-2$ has multiplicity $n-5\geq 3,$ a contradiction.

\end{proof}

\begin{lemma}
\label{cycles}
The cycle graph $C_n$ is a forbidden subgraph of $G$ for $4\leq n \leq 7.$
\end{lemma}

\begin{proof}
We proceed with contradiction. We consider $n=4,5$ first.

Suppose $C_4$ is a subgraph of $G.$ Then $D_G[V(C_4)] = D(C_4).$ So by Lemma~\ref{simple cycles}, $C_4$ is forbidden. 

Suppose $C_5$ is a subgraph of $G.$ Then $D_G[V(C_5)] = D(C_5).$ So by Lemma~\ref{simple cycles}, $C_5$ is forbidden. 

For $n\geq 6,$ suppose $C_n$ is a subgraph of $G.$ For any $1\leq p,q \leq n$, denote the distance between $v_p$ and $v_q$ along $C_n$ by $d(p,q) = \min(|p-q|, n-|p-q|).$ Next, denote the directed distance between $v_p$ and $v_q$ along $C_n$, by $d'(p, q)=q-p \bmod{n}.$ Also denote $S_{pq}$ as the subset of $V(C_n)$ defined as $S_{pq}=\{v_{p+i \bmod{n}} | i\in\mathbb{Z}, 0\leq i\leq d'(p,q)\}$.

Then $D_G[V(C_n)] = 
\{d_{ij}\}_{n\times n},$ where $\begin{cases} d_{ij} = d(i, j), \text{ if } d(i, j) \leq 2, \\ 2\leq d_{ij} \leq d(i, j), \text{ if } d(i, j) > 2.\\ \end{cases}$

If $d_{ij} = d(i, j)$ for every pair $i, j$ then $D_G[V(C_n)]=D(C_n),$ so by Lemma~\ref{simple cycles}, $C_n$ is forbidden.

Otherwise, $d_{ij} = d(i, j)$ does not hold for every pair $i, j$. Then there exists some vertex $x \in V(G)\backslash V(C_n)$ adjacent to two distinct vertices $v_i, v_j \in V(C_n).$

We claim that every $v\in V(C_n)$ is adjacent to $x.$ Suppose otherwise. Then there exists $v_i, v_j$ both adjacent to $x$ such that $d'(i, j) \geq 2$ and every $v\in S_{ij} \backslash \{v_i, v_j\}$ is not adjacent to $x.$ Then $D_G[\{ x \} \cup S_{ij}]=D(C_{d'(i, j)+2}),$ where $d'(i, j)+2\geq 4,$ a contradiction by Lemma~\ref{simple cycles}.

Then $d_{ij} \leq 2$ for all $i,j$, i.e., $D_G[V(C_n)] = 
\{d_{ij}\}_{n\times n},$ where $\begin{cases} d_{ij} = d(i, j) \text{ if } d(i, j) \leq 2, \\ d_{ij} = 2 \text{ if } d(i, j) > 2.\\ \end{cases}$

We check that for $n=6,7,$ $\lambda_5$ of the above matrix is $-3, -1.5550$, respectively, a contradiction by Lemma~\ref{submatrix}.
\end{proof}

\section{Graphs of Diameter 4 or Higher}

In this section we consider graphs cosepctral to $T(a,b)$ with diameter 4 or higher.

\begin{theorem}
\label{S(a',b')}
If $T(a',b')$ is cospectral to $T(a,b)$, then $T(a',b') \cong T(a,b)$.
\end{theorem}

\begin{proof}
For the graph $T(a,b)$ we know that the characteristic polynomial is $(-\lambda-2)^{a+b-2}p_{a,b}(\lambda)$ from Lemma~\ref{double star polynomial}. 

Let us consider $T(a,b)$ and $T(a',b')$. Assume for the sake of contradiction that these have the same distance spectra. Then they must have the same number of distance eigenvalues, so $a+b = a'+b'$. Equating the $\lambda^1$ coefficients of $p_{a,b}(\lambda)$ and $p_{a',b'}(\lambda)$, we must have $40+36a+36b+8ab=40+36a'+36b'+8a'b'$ which implies $ab=a'b'$. Since the unordered pairs $(a,b)$ and $(a', b')$ have the same sum and product, the unordered pairs must be equal.
\end{proof}

\begin{figure}[htbp]
    \centering
    \begin{minipage}{.5\textwidth}
    \begin{tikzpicture}[node distance={15mm}, thin, main/.style = {draw, circle}]
\node[main] (1) {$v_1$}; 
\node[main] (2) [right of=1] {$v_2$};
\node[main] (3) [right of=2] {$v_3$}; 
\node[main] (4) [right of=3] {$v_4$};
\node[main] (5) [right of=4] {$v_5$};
\node[main] (6) [above of=3] {$v_6$};
\draw (1) -- (2);
\draw (2) -- (3);
\draw (3) -- (4);
\draw (4) -- (5);
\draw (1) -- (6);
\draw (2) -- (6);
\draw (3) -- (6);
\draw (4) -- (6);
\draw (5) -- (6);
\end{tikzpicture}
    \caption{$H_1$}
    \label{fig:h1}
    \end{minipage}%
    \begin{minipage}{.5\textwidth}
    \begin{tikzpicture}[node distance={15mm}, thin, main/.style = {draw, circle}]
\node[main] (1) {$v_1$}; 
\node[main] (2) [right of=1] {$v_2$};
\node[main] (3) [right of=2] {$v_3$}; 
\node[main] (4) [right of=3] {$v_4$};
\node[main] (5) [right of=4] {$v_5$}; 
\node[main] (6) [above of=3] {$v_6$};
\draw (1) -- (2);
\draw (2) -- (3);
\draw (3) -- (4);
\draw (4) -- (5);
\draw (1) -- (6);
\draw (2) -- (6);
\draw (3) -- (6);
\draw (4) -- (6);
\end{tikzpicture}
    \caption{$H_2$}
    \label{fig:h2}
    \end{minipage}
\end{figure}

\begin{figure}[htbp]
    \centering
    \begin{minipage}{.5\textwidth}
    \begin{tikzpicture}[node distance={15mm}, thin, main/.style = {draw, circle}]
\node[main] (1) {$v_1$}; 
\node[main] (2) [right of=1] {$v_2$};
\node[main] (3) [right of=2] {$v_3$}; 
\node[main] (4) [right of=3] {$v_4$};
\node[main] (5) [right of=4] {$v_5$}; 
\node[main] (6) [above of=3] {$v_6$};
\draw (1) -- (2);
\draw (2) -- (3);
\draw (3) -- (4);
\draw (4) -- (5);
\draw (2) -- (6);
\draw (3) -- (6);
\draw (4) -- (6);
\end{tikzpicture}
    \caption{$H_3$}
    \label{fig:h3}
    \end{minipage}%
    \begin{minipage}{.5\textwidth}
    \begin{tikzpicture}[node distance={15mm}, thin, main/.style = {draw, circle}]
\node[main] (1) {$v_1$}; 
\node[main] (2) [right of=1] {$v_2$};
\node[main] (3) [right of=2] {$v_3$}; 
\node[main] (4) [right of=3] {$v_4$};
\node[main] (5) [right of=4] {$v_5$}; 
\node[main] (6) [above of=3] {$v_6$};
\draw (1) -- (2);
\draw (2) -- (3);
\draw (3) -- (4);
\draw (4) -- (5);
\draw (1) -- (6);
\draw (2) -- (6);
\draw (3) -- (6);
\end{tikzpicture}
    \caption{$H_4$}
    \label{fig:h4}
    \end{minipage}
\end{figure}

\begin{figure}[htbp]
    \centering
    \begin{minipage}{.5\textwidth}
    \begin{tikzpicture}[node distance={15mm}, thin, main/.style = {draw, circle}]
\node[main] (1) {$v_1$}; 
\node[main] (2) [right of=1] {$v_2$};
\node[main] (3) [right of=2] {$v_3$}; 
\node[main] (4) [right of=3] {$v_4$};
\node[main] (5) [right of=4] {$v_5$}; 
\node[main] (6) [above of=3] {$v_6$};
\draw (1) -- (2);
\draw (2) -- (3);
\draw (3) -- (4);
\draw (4) -- (5);
\draw (2) -- (6);
\draw (3) -- (6);
\end{tikzpicture}
    \caption{$H_5$}
    \label{fig:h5}
    \end{minipage}%
    \begin{minipage}{.5\textwidth}
    \begin{tikzpicture}[node distance={15mm}, thin, main/.style = {draw, circle}]
\node[main] (1) {$v_1$}; 
\node[main] (2) [right of=1] {$v_2$};
\node[main] (3) [right of=2] {$v_3$}; 
\node[main] (4) [right of=3] {$v_4$};
\node[main] (5) [right of=4] {$v_5$}; 
\node[main] (6) [above of=3] {$v_6$};
\draw (1) -- (2);
\draw (2) -- (3);
\draw (3) -- (4);
\draw (4) -- (5);
\draw (1) -- (6);
\draw (2) -- (6);
\end{tikzpicture}
    \caption{$H_6$}
    \label{fig:h6}
    \end{minipage}
\end{figure}

\begin{figure}[htbp]
    \centering
        \begin{minipage}{.5\textwidth}
    \begin{tikzpicture}[node distance={15mm}, thin, main/.style = {draw, circle}]
\node[main] (1) {$v_1$}; 
\node[main] (2) [right of=1] {$v_2$};
\node[main] (3) [right of=2] {$v_3$}; 
\node[main] (4) [right of=3] {$v_4$};
\node[main] (5) [right of=4] {$v_5$}; 
\node[main] (6) [above of=3] {$v_6$};
\draw (1) -- (2);
\draw (2) -- (3);
\draw (3) -- (4);
\draw (4) -- (5);
\draw (3) -- (6);
\end{tikzpicture}
    \caption{$H_7$}
    \label{fig:h7}
    \end{minipage}%
    \begin{minipage}{.5\textwidth}
    \begin{tikzpicture}[node distance={15mm}, thin, main/.style = {draw, circle}]
\node[main] (1) {$v_1$}; 
\node[main] (2) [right of=1] {$v_2$};
\node[main] (3) [right of=2] {$v_3$}; 
\node[main] (4) [right of=3] {$v_4$};
\node[main] (5) [right of=4] {$v_5$}; 
\node[main] (6) [right of=5] {$v_6$};
\draw (1) -- (2);
\draw (2) -- (3);
\draw (3) -- (4);
\draw (4) -- (5);
\draw (5) -- (6);
    \end{tikzpicture}
    \caption{$P_6$}
    \label{fig:p6}
    \end{minipage}
\end{figure}

\begin{lemma}
\label{d4 forbidden}
If $G$ and $T(a,b)$ are cospectral, then $H_1, H_2, H_3, H_4, H_5, H_6,H_7$ and $P_6$ are forbidden subgraphs of $G$.
\end{lemma}
\begin{proof}
We proceed with contradiction. 

Suppose $H_1$ is an induced subgraph of $G.$ Then $D_G[V(H_1)] = D(H_1).$ We check that $\lambda_5(D(H_1)) \neq -2,$ a contradiction by Lemma~\ref{submatrix}.

Suppose $H_2$ is an induced subgraph of $G.$ Then $D_G[V(H_2)] = \begin{pmatrix}0 & 1 & 2 & 2 & a & 1 \\ 1 & 0 & 1 & 2 & b & 1 \\ 2 & 1 & 0 & 1 & 2 & 1 \\ 2 & 2 & 1 & 0 & 1 & 1 \\ a & b & 2 & 1 & 0 & 2 \\ 1 & 1 & 1 & 1 & 2 & 0 \end{pmatrix},$ where $a,b\in\{2,3\}.$ We check that $\lambda_5(D_G[V(H_2)]) \neq -2$ for all combinations of $a,b,$ a contradiction by Lemma~\ref{submatrix}.

Suppose $H_3$ is an induced subgraph of $G.$ Then $D_G[V(H_3)] = \begin{pmatrix}0 & 1 & 2 & a & b & 2 \\ 1 & 0 & 1 & 2 & c & 1 \\ 2 & 1 & 0 & 1 & 2 & 1 \\ a & 2 & 1 & 0 & 1 & 1 \\ b & c & 2 & 1 & 0 & 2 \\ 2 & 1 & 1 & 1 & 2 & 0 \end{pmatrix},$ where $a,c\in\{2,3\}$ and $b\in\{2,3,4\}.$ We check that $\lambda_5(D_G[V(H_3)]) = -2$ only when $(a,b,c) = (3,4,3),$ in which case $\lambda_4(D_G[V(H_3)]) = -1,$ a contradiction by Lemma~\ref{submatrix}.

Suppose $H_4$ is an induced subgraph of $G$. Then $D_G[V(H_4)] = \begin{pmatrix} 0 & 1 & 2 & a & b & 1 \\ 1 & 0 & 1 & 2 & c & 1 \\ 2 & 1 & 0 & 1 & 2 & 1 \\ a & 2 & 1 & 0 & 1 & 2 \\ b & c & 2 & 1 & 0 & d \\ 1 & 1 & 1 & 2 & d & 0 \end{pmatrix},$ where $a,c,d \in \{ 2,3 \}$ and $b \in \{ 2,3,4 \}$. We can check that $\lambda_5(D_G[V(H_4)]) \neq -2$ for all combinations of $a,b,c,d$, a contradiction by Lemma~\ref{submatrix}.

Suppose $H_5$ is an induced subgraph of $G.$ Then $D_G[V(H_5)] = \begin{pmatrix}0 & 1 & 2 & a & b & 2 \\ 1 & 0 & 1 & 2 & c & 1 \\ 2 & 1 & 0 & 1 & 2 & 1 \\ a & 2 & 1 & 0 & 1 & 2 \\ b & c & 2 & 1 & 0 & d \\ 2 & 1 & 1 & 2 & d & 0 \end{pmatrix},$ where $a,c,d\in\{2,3\}$ and $b\in\{2,3,4\}.$ We check that $\lambda_5(D_G[V(H_5)]) \neq -2$ for all combinations of $a,b,c,d$, a contradiction by Lemma~\ref{submatrix}.

Suppose $H_6$ is an induced subgraph of $G$. Then $D_G[V(H_6)] = \begin{pmatrix} 0 & 1 & 2 & a & b & 1 \\ 1 & 0 & 1 & 2 & c & 1 \\ 2 & 1 & 0 & 1 & 2 & 2 \\ a & 2 & 1 & 0 & 1 & d \\ b & c & 2 & 1 & 0 & e \\ 1 & 1 & 2 & d & e & 0 \end{pmatrix}$ where $a, c, d \in \{ 2,3 \}$ and $b, e \in \{ 2, 3, 4 \}$. We check that $\lambda_5(D_G[V(H_6)]) \neq -2$ for all combinations of $a,b,c,d$, a contradiction by Lemma~\ref{submatrix}.

Suppose $H_7$ is an induced subgraph of $G$. Then $D_G[V(H_7)] = \begin{pmatrix} 0 & 1 & 2 & a & b & c \\ 1 & 0 & 1 & 2 & d & 2 \\ 2 & 1 & 0 & 1 & 2 & 1 \\ a & 2 & 1 & 0 & 1 & 2 \\ b & d & 2 & 1 & 0 & e \\ c & 2 & 1 & 2 & e & 0 \end{pmatrix}$ where $a, c, d, e \in \{ 2,3 \}$ and $b \in \{ 2, 3, 4 \}$. We check that $\lambda_5(D_G[V(H_7)]) = -2$ only when $(a,b,c,d,e) = (3,4,2,3,2)$. But $c=2$ implies there exists $x \in V(G)$ such that $x$ is adjacent to $v_1$ and $v_6$. If $x$ is not adjacent to $v_2$ and not adjacent to $v_3,$ then $G[\{x,v_1,v_2,v_3,v_6\}]=C_5.$ If $x$ is adjacent to $v_2$ but not $v_3$, then $G[\{x,v_2,v_3,v_6\}] = C_4$, and if $x$ is adjacent to $v_3$ but not $v_2$, then $G[\{x,v_1,v_2,v_3\}] = C_4.$ Finally, if $x$ is adjacent to both $v_2$ and $v_3,$ then either $G[\{x,v_1,v_2,v_3,v_4,v_5\}]=H_4,H_2$ or $H_1$, or $x$ is adjacent to $v_5$ but not $v_4,$ in which case $G[\{x,,v_3,v_4,v_5\}]=C_4.$ All cases lead to contradictions by Lemmas~\ref{cycles} and \ref{d4 forbidden}.

Suppose $P_6$ is an induced subgraph of $G$. We have $D_G[V(P_6)]=\begin{pmatrix} 0 & 1 & 2 & a & b & c \\
                                1 & 0 & 1 & 2 & d & e \\
                                2 & 1 & 0 & 1 & 2 & f \\
                                a & 2 & 1 & 0 & 1 & 2 \\
                                b & d & 2 & 1 & 0 & 1 \\
                                c & e & f & 2 & 1 & 0 \\
\end{pmatrix},$ where $a,d,f \in \{2,3\},$ $b,e \in \{2,3,4\}, $ and $c \in \{2,3,4,5\}$. Note that if $c=5,$ then $a=d=f=3$ and $b=e=4,$ and if $c=4,$ then $b,e \in \{3,4\}.$ By Lemma~\ref{submatrix}, the subgraph must have $\lambda_5 = -2.$ Checking each of the cases, only $(a,b,c,d,e,f)=(2,3,4,3,3,2)$ satisfies this. However, $a=2$ implies that there exists $x\in V(G)$ such that $v_1$ and $v_4$ are both adjacent to $x$. If $x$ is not adjacent to both $v_2$ and $v_3,$ then $G[\{x,v_1,v_2,v_3,v_4\}] = C_5.$ If $x$ is adjacent to $v_2$ but not $v_3,$ then $G[\{x,v_2,v_3,v_4\}] = C_4,$ and if $x$ is adjacent to $v_3$ but not $v_2,$ then $G[\{x,v_1,v_2,v_3\}] = C_4.$ If $x$ is adjacent to both $v_2$ and $v_3,$ then $G[\{x,v_1,v_2,v_3,v_4,v_5\}]=H_1$ or $H_2.$ All cases lead to contradictions.

\end{proof}

Since any graph of diameter at least 5 contains $P_6$ as a subgraph, this proves the following theorem.

\begin{theorem}
No graph $G$ cospectral to $T(a,b)$ can have diameter greater than 4.
\end{theorem}

For the rest of this section, we consider the case where $G$ has diameter $4.$

By the definition of a diameter 4 graph, there exist two vertices $x_1, x_5 \in V(G)$ such that $d_{x_1x_5}=4$. Suppose that $P=x_1 x_2 x_3 x_4 x_5$ is a path with length 4 in $G$. Let $X=\left\{x_1, x_2, x_3, x_4, x_5\right\}$, then $G[X]=P_{5}$. Denote by $V_{i}(i=0,1,2,3,4,5)$ the vertex subset of $V \backslash X$, consisting of vertices adjacent to $i$ vertices of $X$. Clearly $V \backslash X=\bigcup_{i=0}^{5} V_{i}.$

\begin{lemma}
\label{V2 to V5}
The sets $V_2, V_3, V_4,$ and $V_5$ are empty.
\end{lemma}

\begin{proof}
First, we prove $V_5$ is empty. Suppose otherwise, so there exists $v\in V_5.$ Then $x_1vx_5$ is a path of length $2$ from $x_1$ to $x_5,$ contradicting that $d_{x_1x_5}=4$.

Next, we prove $V_4$ is empty. Suppose otherwise, so there exists $v\in V_4.$ Then $v$ is adjacent to at least one of $x_1$ and $x_5$. Without loss of generality, $v$ is adjacent to $x_1.$ Either $v$ is adjacent to $x_5,$ so $x_1vx_5$ is a path of length $2,$ or $v$ is not adjacent to $x_5$ and $x_1vx_4x_5$ is a path of length $3.$ Both contradict that $d_{x_1x_5}=4$.

Next, we prove $V_3$ is empty. Suppose otherwise, so there exists $v\in V_3.$ There exists $1\leq a<b \leq 5$ such that $v$ is adjacent to $x_a, x_b$ and not adjacent to $x_i$ for any $a < i < b.$ Pick some $a,b$ so that $b-a$ is maximal. Then $G[\{v\}\cup\{x_{j} | a\leq j\leq b\}] = C_{b-a+2}.$ If $b-a\geq 2,$ since $4\leq b-a+2\leq 6$, we have a contradiction by Lemma~\ref{cycles}. Otherwise, $b=a+1,$ so $G[\{v\}\cup X] = H_3$ or $H_4,$ a contradiction by Lemma~\ref{d4 forbidden}.

Lastly, we prove $V_2$ is empty. Suppose otherwise, so there exists $v\in V_2.$ Let $v$ be adjacent to $x_a$ and $x_b$, where $a<b.$ Then $G[\{v\}\cup\{x_{j} | a\leq j\leq b\}] = C_{b-a+2}.$ If $b-a\geq 2,$ then $4\leq b-a+2\leq 6$, and we have a contradiction by Lemma~\ref{cycles}. Otherwise, $b=a+1,$ so $G[\{v\}\cup X] = H_5$ or $H_6,$ a contradiction by Lemma~\ref{d4 forbidden}.

Therefore, we have that $V_2, V_3, V_4, V_5$ are all empty.
\end{proof}

\begin{lemma}
\label{V1}
For all $v\in V_1,$ $v$ is adjacent to either $x_2$ or $x_4.$ Furthermore, for $u,v \in V_1,$ $u$ cannot be adjacent to $v.$
\end{lemma}

\begin{proof}
Suppose for some $v\in V_1$ that $v$ is adjacent to $x_1$ or $x_5$. Then $G[\{v\}\cup X] = P_6,$ a contradiction by Lemma~\ref{d4 forbidden}.

Suppose for some $v\in V_1$ that $v$ is adjacent to $x_3.$ Then $G[\{v\}\cup X] = H_7,$ which is a contradiction by Lemma~\ref{d4 forbidden}.

Thus for all $v \in V_1,$ $v$ is adjacent to either $x_2$ or $x_4.$

Consider $u,v \in V_1.$ Without loss of generality $u$ is adjacent to $x_2.$ Suppose now that $u$ is adjacent to $v.$

If $v$ is adjacent to $x_2,$ then $G[\{u,v,x_2,x_3,x_4,x_5\}]=H_6$, a contradiction by Lemma~\ref{d4 forbidden}.

Otherwise, $v$ is adjacent to $x_4$. Since $u,v \in V_1,$ we have $G[\{v,u,x_2,x_3,x_4\}] = C_5,$ a contradiction by Lemma~\ref{cycles}. 

Thus, $u$ is not adjacent to $v.$
\end{proof}

\begin{lemma}
\label{V0}
The set $V_0$ is empty.
\end{lemma}

\begin{proof}
If $V_0$ is not empty, some $v\in V_0$ must be adjacent to some $u\in V_1$ because all $V_k$ for $k > 1$ are empty. Without loss of generality, let $u$ be adjacent to $x_2$. Then $G[\{v,u,x_2,x_3,x_4,x_5\}]=P_6$, so by Lemma~\ref{d4 forbidden}, we have a contradiction.

Thus $V_0$ is empty.
\end{proof}

\begin{theorem}
\label{d4 full}
If $G$ and $T(a,b)$ are cospectral and $G$ has diameter $4,$ then $G \cong T(a,b).$
\end{theorem}

\begin{proof}
If $G$ satisfies Lemmas~\ref{V2 to V5}, \ref{V1}, and \ref{V0}, then $G=T(a',b')$ for some $a',b'$. By Theorem~\ref{S(a',b')}, we are done.
\end{proof}

We now move on to graphs of lower diameter.

\section{Diameter 3 Graphs}
In this section we consider graphs $G$ cospectral to $T(a,b)$ with diameter $3.$

There exists vertices $x_1,x_4$ such that $d_{x_1,x_4}=3.$ Suppose $x_1x_2x_3x_4$ is a path of length 3 in $G.$ Denote $X=\{x_1,x_2,x_3,x_4\}.$ Note that $G[X]=P_4$.

\begin{figure}[!h]
    \centering
    \begin{minipage}{.4\textwidth}
    \begin{tikzpicture}[node distance={15mm}, thin, main/.style = {draw, circle}]
\node[main] (1) {$v_1$}; 
\node[main] (2) [right of=1] {$v_2$};
\node[main] (3) [right of=2] {$v_3$}; 
\node[main] (4) [right of=3] {$v_4$};
\node[main] (5) [above right of=2] {$v_5$}; 
\draw (1) -- (2);
\draw (2) -- (3);
\draw (3) -- (4);
\draw (5) -- (1);
\draw (5) -- (2);
\draw (5) -- (3);
\draw (5) -- (4);
\end{tikzpicture}
    \caption{$F_1$}
    \label{fig:f1}
    \end{minipage}%
    \begin{minipage}{.4\textwidth}
    \begin{tikzpicture}[node distance={15mm}, thin, main/.style = {draw, circle}]
\node[main] (1) {$v_1$}; 
\node[main] (2) [right of=1] {$v_2$};
\node[main] (3) [right of=2] {$v_3$}; 
\node[main] (4) [right of=3] {$v_4$};
\node[main] (5) [above right of=2] {$v_5$}; 
\draw (1) -- (2);
\draw (2) -- (3);
\draw (3) -- (4);
\draw (1) -- (5);
\draw (2) -- (5);
\draw (3) -- (5);
\end{tikzpicture}
    \caption{$F_2$}
    \label{fig:f2}
    \end{minipage}
\end{figure}
\begin{figure}[!h]
    \centering
    \begin{minipage}{.4\textwidth}
    \begin{tikzpicture}[node distance={15mm}, thin, main/.style = {draw, circle}]
\node[main] (1) {$v_1$}; 
\node[main] (2) [right of=1] {$v_2$};
\node[main] (3) [right of=2] {$v_3$}; 
\node[main] (4) [right of=3] {$v_4$};
\node[main] (5) [above right of=2] {$v_5$};
\draw (1) -- (2);
\draw (2) -- (3);
\draw (3) -- (4);
\draw (1) -- (5);
\draw (2) -- (5);
\end{tikzpicture}
    \caption{$F_3$}
    \label{fig:f3}
    \end{minipage}%
    \begin{minipage}{.4\textwidth}
    \hspace{6em}\begin{tikzpicture}[node distance={15mm}, thin, main/.style = {draw, circle}]
\node[main] (1) {$v_1$}; 
\node[main] (2) [right of=1] {$v_2$};
\node[main] (3) [above of=2] {$v_3$}; 
\node[main] (4) [above of=1] {$v_4$};
\draw (1) -- (2);
\draw (1) -- (3);
\draw (1) -- (4);
\draw (2) -- (3);
\draw (2) -- (4);
\draw (3) -- (4);
\end{tikzpicture}
    \caption{$K_4$}
    \label{fig:k4}
    \end{minipage}
\end{figure}
\begin{figure}[!h]
    \centering
    \begin{tikzpicture}[node distance={20mm}, thin, main/.style = {draw, circle}]
\node[main] (1) {$v_1$}; 
\node[main] (2) [right of=1] {$v_2$};
\node[main] (3) [right of=2] {$v_3$}; 
\node[main] (4) [right of=3] {$v_4$};
\node[main] (5) [above left of=1] {$v_5$};
\node[main] (6) [right of=5] {$v_6$};
\node[main] (7) [right of=6] {$v_7$};
\node[main] (8) [right of=7] {$v_8$};
\node[main] (9) [right of=8] {$v_9$};
\node[main] (10) [right of=9] {$v_{10}$};
\draw (1) -- (2);
\draw (2) -- (3);
\draw (3) -- (4);
\draw (2) -- (5);
\draw (3) -- (5);
\draw (2) -- (6);
\draw (3) -- (6);
\draw (2) -- (7);
\draw (3) -- (7);
\draw (2) -- (8);
\draw (3) -- (8);
\draw (2) -- (9);
\draw (3) -- (9);
\draw (2) -- (10);
\draw (3) -- (10);
\end{tikzpicture}
    \caption{$F_4$}
    \label{fig:f7}
\end{figure}

\begin{lemma}
\label{d3 forbidden}
If $G$ and $T(a,b)$ are cospectral, then $F_1, F_2, F_3, K_4$ and $F_4$ are forbidden.
\end{lemma}

\begin{proof}
Suppose $F_1$ is an induced subgraph of $G$. Then $D_G[V(F_1)] = \begin{pmatrix} 0 & 1 & 2 & a & 1 \\ 1 & 0 & 1 & 2 & 1 \\ 2 & 1 & 0 & 1 & 1 \\ a & 2 & 1 & 0 & 1 \\ 1 & 1 & 1 & 1 & 0 \end{pmatrix}$ where $a \in \{ 2, 3 \}$. In both cases, we have a contradiction by Lemma~\ref{submatrix}.

Suppose $F_2$ is an induced subgraph of $G$. Then $D_G[V(F_2)] = \begin{pmatrix} 0 & 1 & 2 & a & 1 \\ 1 & 0 & 1 & 2 & 1 \\ 2 & 1 & 0 & 1 & 1 \\ a & 2 & 1 & 0 & 2 \\ 1 & 1 & 1 & 2 & 0 \end{pmatrix}$ where $a \in \{ 2,3 \}$. Again, in both cases we have a contradiction by Lemma~\ref{submatrix}.

Suppose $F_3$ is an induced subgraph of $G$. Then $D_G[V(F_3)] = \begin{pmatrix} 0 & 1 & 2 & a & 1 \\ 1 & 0 & 1 & 2 & 1 \\ 2 & 1 & 0 & 1 & 2 \\ a & 2 & 1 & 0 & b \\ 1 & 1 & 2 & b & 0 \end{pmatrix}$ where $a,b \in \{ 2,3 \}$. We get that if $(a,b) = (3,2)$ or $(a,b) = (3,3)$, we have a contradiction by Lemma~\ref{submatrix}. Both of the other cases have $a = 2$, which implies that there is some vertex $d$ such that $v_1$ and $v_4$ are both adjacent to $d$. However, then one of the following four is true: $G[\{ v_1, v_2, v_3, v_4, d \}] = C_5$; $G[\{ v_1, v_2, v_3, d \}] = C_4$; $G[\{ v_2, v_3, v_4, d \}] = C_4$; $G[\{d, v_1, v_2, v_3, v_4\}] = F_1.$ All cases lead to contradictions.

Suppose $K_4$ is an induced subgraph of $G$. Then $D_G[V(K_4)] = D(K_4) =  \begin{pmatrix} 0 & 1 & 1 & 1 \\ 1 & 0 & 1 & 1 \\ 1 & 1 & 0 & 1 \\ 1 & 1 & 1 & 0 \end{pmatrix}$. We get that $\lambda_4(K_4) = -1,$ a contradiction by Lemma~\ref{submatrix}.

Suppose $F_4$ is an induced subgraph of $G$. Then $D_G[V(F_4)] =  \begin{pmatrix} 0 & 1 & 2 & a & 2 & 2 & 2 & 2 & 2 & 2 \\ 1 & 0 & 1 & 2 & 1 & 1 & 1 & 1 & 1 & 1 \\ 2 & 1 & 0 & 1 & 1 & 1 & 1 & 1 & 1 & 1 \\ a & 2 & 1 & 0 & 2 & 2 & 2 & 2 & 2 & 2 \\ 2 & 1 & 1 & 2 & 0 & 2 & 2 & 2 & 2 & 2 \\ 2 & 1 & 1 & 2 & 2 & 0 & 2 & 2 & 2 & 2 \\ 2 & 1 & 1 & 2 & 2 & 2 & 0 & 2 & 2 & 2 \\ 2 & 1 & 1 & 2 & 2 & 2 & 2 & 0 & 2 & 2 \\ 2 & 1 & 1 & 2 & 2 & 2 & 2 & 2 & 0 & 2 \\ 2 & 1 & 1 & 2 & 2 & 2 & 2 & 2 & 2 & 0 \end{pmatrix},$ where $a\in \{2,3\}.$ If $a=2$,  we find $\lambda_9 \neq -2.$ If $a=3,$ $\lambda_2 = 0.$ Both contradict Lemma~\ref{submatrix}.
\end{proof}

Denote by $V_{i}(i=0,1,2,3,4)$ the vertex subset of $V \backslash X$, consisting of vertices adjacent to $i$ vertices of $X$. Clearly $V \backslash X=\bigcup_{i=0}^{4} V_{i} .$

\begin{lemma}
\label{leaves}
If $v_1, v_2, v_3, v_4$ are vertices of $G$ with $G[v_1,v_2,v_3,v_4]=P_4,$ then $v_1$ and $v_4$ are leaves, namely, they have degree $1.$
\end{lemma}

\begin{proof}
Assume for the sake of contradiction that some vertex $u$ is adjacent to $v_1.$ 

Suppose $u$ is adjacent to some vertex in $\{v_2,v_3,v_4\}.$ Then, among all pairs $(i,j), i<j$ with $u$ adjacent to both $v_i$ and $v_j$ and not adjacent to $v_k$ for $i< k < j,$ consider the pair $(i,j)$ such that $j-i$ is maximal. If $j-i\geq2,$ then $G[u,\{v_k | i\leq k\leq j\}]=C_{j-i+2},$ where $5\geq j-i+2\geq 4,$ a contradiction by Lemma~\ref{cycles}. 

Thus $j-i=1.$ But $G[\{u,v_1,v_2,v_3,v_4\}]=F_1, F_2,$ or $F_3,$ all contradictions. 

So $u$ cannot be adjacent to any vertex in $\{v_2,v_3,v_4\}.$ Then $G[\{u,v_1,v_2,v_3,v_4\}]=P_5$. Since $uv_1v_2v_3v_4$ is a path of length $4$ in $G$, we take $X=\{u,v_1,v_2,v_3,v_4\}$ and proceed as in the Diameter 4 Graphs section, using Lemmas~\ref{V2 to V5}, \ref{V1}, and \ref{V0} to conclude $G\cong T(a,b)$, contradicting that $G$ has diameter $3$.

Thus the assumption that $u$ is adjacent to $v_1$ is false, so $v_1$ is a leaf. Similarly, $v_4$ is also a leaf.
\end{proof}

\begin{lemma}
\label{d3 V0-V4}
We have the following:
\begin{enumerate}[label=(\alph*)]
    \item Sets $V_3$ and $V_4$ are empty.
    \label{d3 V3-V4}
    \item For any $v \in V_2,$ $v$ must be adjacent to $x_2$ and $x_3.$ For any $u \in V_1,$ $u$ must be adjacent $x_2$ or $x_3.$
    \label{d3 V1-V2}
    \item For any $v \in V_0,$ $v$ is not adjacent to any $u\in V_1$. Further, $v$ must be adjacent to some $u\in V_2.$
    \label{d3 V0}
\end{enumerate}
\end{lemma}

\begin{proof}
Suppose $v\in V_4.$ Then $d_{x_1x_4}=2$ a contradiction. Thus $V_4$ is empty.

Suppose $v\in V_3.$ Then $v$ is adjacent to at least one of $x_1,x_4$. But $x_1$ and $x_4$ are both leaves by Lemma~\ref{leaves}. Thus $V_3$ is empty. This proves part~\ref{d3 V3-V4}.

Part~\ref{d3 V1-V2} follows from that $x_1$ and $x_4$ are leaves.

Now we prove part~\ref{d3 V0}. Consider $v\in V_0.$ Define $d(v,X)$ as the minimum distance from $v$ to a vertex in $X.$ 

We claim $d(v, X) < 3.$ Suppose $d(v, X) \geq 3.$ Recall that $x_4$ only has one neighbor, namely $x_3.$ Thus any path from $v$ to $x_4$ must pass through $x_3$. As usual, denote $d_{ab}$ as the distance between vertices $a$ and $b.$ We have $d_{vx_4} = d_{vx_3}+1 \geq d(v, X)+1 \geq 4,$ contradicting that $G$ has diameter $3.$

Since $v\in V_0, d(v, X)>1.$ Thus $d(v, X)=2.$ 

By part~\ref{d3 V3-V4}, $v$ is adjacent to some $u\in V_1$ or $V_2.$

Without loss of generality, $u$ is adjacent to $x_2.$ We have that $G[\{u,x_2,x_3,x_4\}]=P_4$ so $u$ is a leaf by Lemma~\ref{leaves}. So $v$ cannot be adjacent to any $u \in V_1.$ It follows that $v$ must be adjacent to some $u\in V_2.$
\end{proof}

We now know about the elements of each $V_i$ for $0\leq i \leq 4$. We move on to looking at possible edges between two elements of $\bigcup_{i=0}^4 V_i$. 

\begin{lemma}
\label{d3 edges}
We have the following:
\begin{enumerate}[label=(\alph*)]
    \item For any two $u,v \in V_2,$ $u$ is not adjacent to $v.$
    \label{d3 V2 edges}
    \item For any $u \in V_0,$ $u$ is not adjacent to any $v\in V_0.$ Further, $u$ is adjacent to exactly one $v\in V_2.$
    \label{d3 V0 edges}
    \item For any $u \in V_1$ and $v \in V_0\cup V_1\cup V_2$, $u$ is not adjacent to $v.$
    \label{d3 V1 edges}
\end{enumerate}
\end{lemma}

\begin{proof}
Consider $u,v \in V_2.$ Suppose $u$ is adjacent to $v.$ By Lemma~\ref{d3 V0-V4}~\ref{d3 V1-V2}, $u$ and $v$ are adjacent to both $x_2$ and $x_3.$ Then $G[\{x_2,x_3,u,v\}] = K_4,$ a contradiction. This proves part~\ref{d3 V2 edges}.

Consider $u\in V_0$. By Lemma~\ref{d3 V0-V4}~\ref{d3 V0}, $u$ is adjacent to some $u' \in V_2.$ Since $G[\{x_1,x_2,u',u\}]=P_4,$ by Lemma~\ref{leaves}, $u$ is a leaf. This proves part~\ref{d3 V0 edges}.

Consider $u \in V_1$ and $v \in V_0\cup V_1\cup V_2$. Without loss of generality $u$ is adjacent to $x_2$. We have that $G[\{u,x_2,x_3,x_4\}]=P_4$ so $u$ is a leaf by Lemma~\ref{leaves}. This proves part~\ref{d3 V1 edges}.
\end{proof}

\begin{remark}
\label{hat summary}
To summarize, we know the following about $G$. Elements of $V_1$ are adjacent to $x_2$ or $x_3$ and elements of $V_2$ are adjacent to $x_2$ and $x_3.$ In addition, each element of $V_0$ is adjacent to exactly one element of $V_2.$ Further, the vertices and edges described above are the \emph{only} possible vertices and edges of $G.$
\end{remark}

We call an element of $V_2$ a \textit{hat,} e.g., if $G$ has $3$ elements in $V_2,$ we say $G$ has $3$ hats.

\begin{lemma}
\label{six hats is illegal}
The graph $G$ has at most $5$ hats.
\end{lemma}

\begin{proof}
Suppose $V_2$ contains $6$ or more elements. Using Lemma~\ref{d3 edges}~\ref{d3 V2 edges}, we see that $F_4$ must be an induced subgraph of $G,$ a contradiction.
\end{proof}

\begin{lemma}
\label{zero hat}
If $G$ has no hats, then $G$ is not cospectral to $T(a,b).$
\end{lemma}

\begin{proof}
If $G$ has no hat, $G$ must be a double star. By the results in \cite{xue}, $G$ is not cospectral to $T(a,b).$
\end{proof}

\begin{lemma}
\label{1 hat}
If $G$ contains exactly 1 hat, then $G$ is not cospectral to $T(a,b)$.
\end{lemma}

\begin{proof}
We proceed by contradiction.

Denote the vertex in $V_2$ by $v.$

Let $a'$ be the number of non-hat vertices adjacent to $x_2$ (including $x_1$), $b'$ be the number of non-hat vertices adjacent to $x_3$ (including $x_4$), and $c'$ be the number of vertices adjacent to $v$ other than $x_2$ and $x_3$.

The characteristic polynomial of $G$ is $$(-\lambda-2)^{a'+b'+c'-3}\det\begin{vmatrix}-\lambda & 1 & 1 & a' & 2b' & 2c' \\ 1 & -\lambda & 1 & 2a' & b' & 2c' \\ 1 & 1 & -\lambda & 2a' & 2b' & c' \\ 1 & 2 & 2 & 2a'-2-\lambda & 3b' & 3c' \\ 2 & 1 & 2 & 3a' & 2b'-2-\lambda & 3c' \\ 2 & 2 & 1 & 3a' & 3b' & 2c'-2-\lambda \end{vmatrix}.$$ Simplifying gives $(-\lambda -2)^{a'+b'+c'-3} p_1(\lambda),$ see Appendix. For this to be equal to the characteristic polynomial of $T(a,b)$, $p_1(\lambda)$ must be divisible by $-\lambda-2$, so $p_1(-2) = 0$. Plugging this in gives $28a'b'c' = 0$. Since $x_2$ is adjacent to $x_1$ and $x_3$ is adjacent to $x_4$, we know $a', b' \geq 1$, so $c' = 0$. We plug $c'=0$ into $p_1(\lambda)$ and divide by $-\lambda-2$ to obtain $q_1$, see Appendix. Since the number of vertices in the two graphs must be equal, we know $a'+b'+3 = a+b+3$. We know $q_1(\lambda)$ and $p(\lambda)$ are equivalent polynomials, so we set their constant terms equal and substitute in the previous equation, giving us that $a'+b' = -4$, which is clearly a contradiction.
\end{proof}

\begin{lemma}
\label{2 to 5 hats}
If $G$ contains between 2 and 5 hats inclusive, then $G$ is not cospectral to $T(a,b)$.
\end{lemma}

\begin{proof}
We proceed by contradiction.

First, we show no vertex in $V_2$ is adjacent to any vertices in $V_0.$

Suppose some vertex $v\in V_2$ is adjacent to some vertex in $u\in V_0.$ Then there exists some other vertex $v' \in V_2,$ distinct from $v.$ We have $G[v',x_3,v,u,x_2]=F_2,$ a contradiction by Lemma~\ref{d3 forbidden}. 

Let $a'$ be the number of non-hat vertices adjacent to $x_2$ (including $x_1$) and $b'$ be the number of non-hat vertices adjacent to $x_3$ (including $x_4$). We showed above that each vertex in $V_2$ has only $x_2$ and $x_3$ as neighbors. Denote by $n$ the number of vertices of $G$.

If $G$ has two hats, its characteristic polynomial is $$(-\lambda-2)^{a'+b'-2}\det\begin{vmatrix}-\lambda & 1 & 1 & 1 & a' & 2b' \\ 1 & -\lambda & 1 & 1 & 2a' & b' \\ 1 & 1 & -\lambda & 2 & 2a' & 2b' \\ 1 & 1 & 2 & -\lambda & 2a' & 2b' \\ 1 & 2 & 2 & 2 & 2a'-2-\lambda & 3b' \\ 2 & 1 & 2 & 2 & 3a' & 2b'-2-\lambda \end{vmatrix}.$$ Computing the determinant gives $(-\lambda-2)^{a'+b'-2} p_2(\lambda)=(-\lambda-2)^{n-6} p_2(\lambda),$ see Appendix. We equate this polynomial to the characteristic polynomial of $T(a,b),$ namely, $(-\lambda-2)^{n-5}p(\lambda).$ Then $p_2(\lambda)$ must contain a factor of $(-\lambda-2)$. Define $q_2$ such that $p_2(\lambda) = (-\lambda-2) q_2(\lambda)$. We also know that $a'+b'+4 = n = a+b+3$, so setting the constant terms of $q_2(\lambda)$ and $p(\lambda)$ equal gives us $a'+b' = -4$ again, which is a contradiction.

If $G$ has three hats, its characteristic polynomial is $$(-\lambda-2)^{a'+b'-2}\det\begin{vmatrix}-\lambda & 1 & 1 & 1 & 1 & a' & 2b' \\ 1 & -\lambda & 1 & 1 & 1 & 2a' & b' \\ 1 & 1 & -\lambda & 2 & 2 & 2a' & 2b' \\ 1 & 1 & 2 & -\lambda & 2 & 2a' & 2b' \\ 1 & 1 & 2 & 2 & -\lambda & 2a' & 2b' \\ 1 & 2 & 2 & 2 & 2 & 2a'-2-\lambda & 3b' \\ 2 & 1 & 2 & 2 & 2 & 3a' & 2b'-2-\lambda \end{vmatrix}.$$ Computing the determinant gives $(-\lambda-2)^{a'+b'-2} p_3(\lambda)=(-\lambda-2)^{n-7} p_3(\lambda),$ see Appendix. We equate this polynomial to the characteristic polynomial of $T(a,b),$ namely, $(-\lambda-2)^{n-5}p(\lambda).$ Then $p_3(\lambda)$ must contain $2$ factors of $(-\lambda-2)$. Define $q_3$ such that $p_3(\lambda) = (-\lambda-2)^2 q_3(\lambda)$. We also know that $a'+b'+5 = n = a+b+3$, so setting the constant terms of $q_3(\lambda)$ and $p(\lambda)$ equal gives us $a'+b' = -4$ again, which is a contradiction.

If $G$ has four hats, its characteristic polynomial is $$(-\lambda-2)^{a'+b'-2}\det\begin{vmatrix}-\lambda & 1 & 1 & 1 & 1 & 1 & a' & 2b' \\ 1 & -\lambda & 1 & 1 & 1 & 1 & 2a' & b' \\ 1 & 1 & -\lambda & 2 & 2 & 2 & 2a' & 2b' \\ 1 & 1 & 2 & -\lambda & 2 & 2 & 2a' & 2b' \\ 1 & 1 & 2 & 2 & -\lambda & 2 & 2a' & 2b' \\ 1 & 1 & 2 & 2 & 2 & -\lambda & 2a' & 2b' \\ 1 & 2 & 2 & 2 & 2 & 2 & 2a'-2-\lambda & 3b' \\ 2 & 1 & 2 & 2 & 2 & 2 & 3a' & 2b'-2-\lambda \end{vmatrix}.$$ Computing the determinant gives $(-\lambda-2)^{a'+b'-2} p_4(\lambda)=(-\lambda-2)^{n-8} p_4(\lambda),$ see Appendix. We equate this polynomial to the characteristic polynomial of $T(a,b),$ namely, $(-\lambda-2)^{n-5}p(\lambda).$ Then $p_4(\lambda)$ must contain $3$ factors of $(-\lambda-2)$. Define $q_4$ such that $p_4(\lambda) = (-\lambda-2)^3 q_4(\lambda)$. We also know that $a'+b'+6 = n = a+b+3$, so setting the constant terms of $q_4(\lambda)$ and $p(\lambda)$ equal gives us $a'+b' = -4$ again, which is a contradiction.

If $G$ has five hats, its characteristic polynomial is $$(-\lambda-2)^{a'+b'-2}\det\begin{vmatrix}-\lambda & 1 & 1 & 1 & 1 & 1 & 1 & a' & 2b' \\ 1 & -\lambda & 1 & 1 & 1 & 1 & 1 & 2a' & b' \\ 1 & 1 & -\lambda & 2 & 2 & 2 & 2 & 2a' & 2b' \\ 1 & 1 & 2 & -\lambda & 2 & 2 & 2 & 2a' & 2b' \\ 1 & 1 & 2 & 2 & -\lambda & 2 & 2 & 2a' & 2b' \\ 1 & 1 & 2 & 2 & 2 & -\lambda & 2 & 2a' & 2b' \\ 1 & 1 & 2 & 2 & 2 & 2 & -\lambda & 2a' & 2b' \\ 1 & 2 & 2 & 2 & 2 & 2 & 2 & 2a'-2-\lambda & 3b' \\ 2 & 1 & 2 & 2 & 2 & 2 & 2 & 3a' & 2b'-2-\lambda \end{vmatrix}.$$ Computing the determinant gives $(-\lambda-2)^{a'+b'-2} p_5(\lambda)=(-\lambda-2)^{n-9} p_5(\lambda),$ see Appendix. We equate this polynomial to the characteristic polynomial of $T(a,b),$ namely, $(-\lambda-2)^{n-5}p(\lambda).$ Then $p_5(\lambda)$ must contain $4$ factors of $(-\lambda-2)$. Define $q_5$ such that $p_5(\lambda) = (-\lambda-2)^4 q_5(\lambda)$. We also know that $a'+b'+7 = n = a+b+3$, so setting the constant terms of $q_5(\lambda)$ and $p(\lambda)$ equal gives us $a'+b' = -4$ again, which is a contradiction.
\end{proof}

\begin{theorem}
\label{d3 full}
If $G$ has diameter 3, then $G$ is not cospectral to $T(a,b).$
\end{theorem}

\begin{proof}
Continuing from Remark~\ref{hat summary}, we find by Lemmas~\ref{six hats is illegal}, \ref{zero hat}, \ref{1 hat} and \ref{2 to 5 hats} that $G$ is not cospectral to $T(a,b)$ if it contains more than $5$ hats, between $2$ and $5$ hats, $1$ hat, or $0$ hats, respectively. Thus, no graph $G$ of diameter $3$ can be cospectral to $T(a,b).$
\end{proof}

\section{$T(a,b)$ is Determined by its Distance Spectrum}

For any graph $G$, denote the edge set of $G$ as $E(G)$.
We have the following result, Corollary 2.5 in \cite{xue}.

\begin{lemma}
\label{d2}
Let $G$ be a graph with order $n$ and $d(G) = 2.$ If $G'$ has the same distance spectrum as $G$, then $|E(G)| < |E(G')|$ when $d(G') \geq 3.$
\end{lemma}

Finally, we have our result.

\begin{theorem}
The graph $T(a,b)$ is determined by its distance spectrum.
\end{theorem}

\begin{proof}
Let $G$ be a connected graph with the same distance spectrum as $T(a,b).$ By Lemma~\ref{d4 forbidden}, $P_6$ is a forbidden subgraph of $G$, so $d(G) \leq 4$. Clearly, $d(G) > 1$. Suppose that $d(G) = 2.$ By Lemma~\ref{d2}, we have $|E(G)| < |E(T(a, b))|$, which contradicts the connectivity of $G$. 

Next, Theorem~\ref{d3 full} gives that $d(G)\neq 3.$ So $d(G)=4,$ and by Theorem~\ref{d4 full}, $G\cong T(a,b).$
\end{proof}

\begin{section}{Acknowledgements}
This paper was made possible through the MIT PRIMES program. The project was mentored under Feng Gui, a graduate student studying mathematics at MIT. We would like to thank PRIMES for the opportunity to conduct this research, and Feng Gui for his support in our project. In addition, we would like to specifically thank Kent Yashaw for his helpful feedback on the paper. 
\end{section}

\section*{Appendix}

\vspace{-5mm}

$$p_1(\lambda) = -16 - 12 a' - 12 b' - 12 c' - 48 \lambda - 52 a' \lambda - 52 b' \lambda - 12 a' b' \lambda - 52 c' \lambda - 12 a' c' \lambda - 12 b' c' \lambda- 48 \lambda^2 - 79 a' \lambda^2 - 79 b' \lambda^2$$ $$ - 30 a' b' \lambda^2- 79 c' \lambda^2 - 30 a' c' \lambda^2 - 30 b' c' \lambda^2 - 9 a' b' c' \lambda^2 - 12 \lambda^3 - 54 a' \lambda^3 - 54 b' \lambda^3 - 22 a' b' \lambda^3 - 54 c' \lambda^3 - 22 a' c' \lambda^3 - 22 b' c' \lambda^3$$ $$ - 8 a' b' c' \lambda^3 + 9 \lambda^4- 17 a' \lambda^4 - 17 b' \lambda^4 - 5 a' b' \lambda^4 - 17 c' \lambda^4 - 5 a' c' \lambda^4 - 5 b' c' \lambda^4 + 6 \lambda^5 - 2 a' \lambda^5 - 2 b' \lambda^5 - 2 c' \lambda^5 + \lambda^6.$$
$$q_1(\lambda) = 8 + 6 a' + 6 b' + (20 + 23 a' + 23 b' + 6 a' b') \lambda + (14 + 28 a' + 28 b' + 12 a' b') \lambda^2 + (-1 + 13 a' + 13 b' + 5 a' b') \lambda^3 $$ $$ + (-4 + 2 a' + 2 b') \lambda^4 - \lambda^5.$$

$$p_2(\lambda) = -16 - 8 a' - 8 b' - 64 \lambda - 40 a' \lambda - 40 b' \lambda - 8 a' b' \lambda - 88 \lambda^2 - 
 70 a' \lambda^2 - 70 b' \lambda^2 - 24 a' b' \lambda^2 - 48 \lambda^3 - 52 a' \lambda^3 - 52 b' \lambda^3 - $$ $$
 20 a' b' \lambda^3 - 5 \lambda^4 - 17 a' \lambda^4 - 17 b' \lambda^4 - 5 a' b' \lambda^4 + 4 \lambda^5 - 
 2 a' \lambda^5 - 2 b' \lambda^5 + \lambda^6.$$
$$q_2(\lambda) = 8 + 4 a' + 4 b' + (28 + 18 a' + 18 b' + 4 a' b') \lambda + (30 + 26 a' + 26 b' + 10 a' b') \lambda^2 + (9 + 13 a' + 13 b' + 5 a' b') \lambda^3 $$ $$+ (-2 + 2 a' + 2 b') \lambda^4 - \lambda^5.$$

$$p_3(\lambda) = 32 + 8 a' + 8 b' + 176 \lambda + 60 a' \lambda + 60 b' \lambda + 8 a' b' \lambda + 336 \lambda^2 + 
 150 a' \lambda^2 + 150 b' \lambda^2 + 40 a' b' \lambda^2 + 296 \lambda^3 + 161 a' \lambda^3  $$ $$ + 
 161 b' \lambda^3 + 54 a' b' \lambda^3  + 122 \lambda^4 + 84 a' \lambda^4 + 84 b' \lambda^4 + 
 28 a' b' \lambda^4 + 15 \lambda^5 + 21 a' \lambda^5 + 21 b' \lambda^5 + 5 a' b' \lambda^5 - 4 \lambda^6 + 
 2 a' \lambda^6 + 2 b' \lambda^6 - \lambda^7.$$
$$q_3(\lambda) = 8 + 2 a' + 2 b' + (36 + 13 a' + 13 b' + 2 a' b') \lambda + (46 + 24 a' + 24 b' + 8 a' b') \lambda^2 + (19 + 13 a' + 13 b' + 5 a' b') \lambda^3 $$$$+ (2 a' + 2 b') \lambda^4 - \lambda^5.$$
\newpage
$$p_4(\lambda) = -64 - 448 \lambda - 64 a' \lambda - 64 b' \lambda - 1072 \lambda^2 - 272 a' \lambda^2 - 272 b' \lambda^2 - 48 a' b' \lambda^2 - 1248 \lambda^3 - 416 a' \lambda^3 - 416 b' \lambda^3 - 112 a' b' \lambda^3$$ $$
 - 780 \lambda^4 - 312 a' \lambda^4 - 312 b' \lambda^4 - 96 a' b' \lambda^4 - 252 \lambda^5 - 124 a' \lambda^5 - 
 124 b' \lambda^5 - 36 a' b' \lambda^5 - 29 \lambda^6 - 25 a' \lambda^6 - 25 b' \lambda^6 - 5 a'b' \lambda^6 + 
 4 \lambda^7 $$$$ - 2 a' \lambda^7 - 2 b' \lambda^7 + \lambda^8.$$
$$q_4(\lambda) = 8 + (44 + 8 a' + 8 b') \lambda + (62 + 22 a' + 22 b' + 6 a' b') \lambda^2 + (29 + 13 a' + 13 b' + 5 a' b') \lambda^3 + (2 + 2 a' + 2 b') \lambda^4 - \lambda^5.$$

$$p_5(\lambda) = 128 - 32 a' - 32 b' + 1088 \lambda - 16 a' \lambda - 16 b' \lambda - 32 a' b' \lambda + 3104 \lambda^2 + 
 368 a' \lambda^2 + 368 b' \lambda^2 + 4432 \lambda^3 + 904 a' \lambda^3$$$$ + 904 b' \lambda^3 + 
 160 a' b' \lambda^3 + 3608 \lambda^4 + 950 a' \lambda^4 + 950 b' \lambda^4 + 240 a' b' \lambda^4 +
 1724 \lambda^5 + 539 a' \lambda^5 + 539 b' \lambda^5 + 150 a' b' \lambda^5 + 454 \lambda^6 $$$$ + 
 172 a' \lambda^6 + 172 b' \lambda^6 + 44 a' b' \lambda^6 + 47 \lambda^7 + 29 a' \lambda^7 + 29 b' \lambda^7 + 
 5 a' b' \lambda^7 - 4 \lambda^8 + 2 a' \lambda^8 + 2 b' \lambda^8 - \lambda^9.$$
$$q_5(\lambda) = 8 - 2 a' - 2 b' + (52 + 3 a' + 3 b' - 2 a' b') \lambda + (78 + 20 a' + 20 b' + 4 a' b') \lambda^2 + (39 + 13 a' + 13 b' + 5 a' b') \lambda^3 $$$$ + (4 + 2 a' + 2 b') \lambda^4 - \lambda^5.$$


\begin{thebibliography}{}

\bibitem{atik}
F. Atik, P. Panigrahi, On the distance spectrum of distance regular graphs, \textit{Linear Algebra Appl.} 478 (2015) 256–273.

\bibitem{fowler}
P.W. Fowler, G. Caporossi, P. Hansen, Distance matrices, Wiener indices, and related invariants of fullerenes, \textit{J. Phys. Chem. A} 105 (2001) 6232–6242.

\bibitem{jin} 
Y.L. Jin, X.D. Zhang, Complete multipartite graphs are determined by their distance spectra, \textit{Linear Algebra Appl.} 448 (2014) 285–291.

\bibitem{lin} 
H.Q. Lin, Y. Hong, J.F. Wang, J.L. Shu, On the distance spectrum of graphs, \textit{Linear Algebra Appl.} 439 (2013) 1662–1669.

\bibitem{lu} 
L. Lu, Q.Y. Huang, X.Y. Huang, The graphs with exactly two distance eigenvalues different from $-1$ and $-3$, \textit{J. Algebraic Comb.} 45 (2017) 629–647.

\bibitem{vandam}
E.R. van Dam, W.H. Haemers, Which graphs are determined by their spectra?, \textit{Linear Algebra Appl.} 373 (2003) 241–272.

\bibitem{koolen}
E.R. van Dam, W.H. Haemers, J.H. Koolen, Cospectral graphs and the generalized adjacency matrix, \textit{Linear Algebra Appl.} 423 (2007) 33–41.

\bibitem{xue}
Jie Xue, Ruifang Liu, Huicai Jia. On the Distance Spectrum of Trees. \textit{Filomat}, vol. 30, no. 6, (2016) 1559–1565. 

\bibitem{zhang}
X. Zhang, Graphs with few distinct $D$-eigenvalues determined by their $D$-spectra, \textit{Linear Algebra Appl.} 628 (2021) 42-55.

\end{thebibliography}
\end{document}